\documentclass[12pt]{amsart}
\usepackage{amsfonts,amssymb,amscd,amstext,mathrsfs,color, comment}

\input xy
\xyoption{all}

\newtheorem{theorem}{Theorem}[section]
\newtheorem{proposition}[theorem]{Proposition}
\newtheorem{lemma}[theorem]{Lemma}
\newtheorem{corollary}[theorem]{Corollary}
\newtheorem*{theorem*}{Main Theorem}

\theoremstyle{definition}

\newtheorem{example}[theorem]{Example}
\newtheorem{remark}[theorem]{Remark}
\numberwithin{equation}{section}

\renewcommand{\L}{{\mathscr L}}

\newcommand{\C}{{\mathbb C}}
\newcommand{\N}{{\mathscr N}}
\renewcommand{\O}{{\mathscr O}}
\newcommand{\F}{{\mathscr F}}

\newcommand{\Z}{{\mathbb Z}}
\newcommand{\Aut}{{\operatorname{Aut}}}
\newcommand{\Hom}{{\operatorname{Hom}}}

\newcommand{\Ker}{{\operatorname{Ker}\,}}
\newcommand{\inv}{{^{-1}}}
\newcommand{\git}{/\!\!/}
\newcommand{\pr}{{\operatorname{pr}}}
\renewcommand{\phi}{\varphi}
\newcommand{\GL}{\operatorname{GL}}

\newcommand{\ind}{\operatorname{ind}}
\newcommand{\res}{\operatorname{res}}
\newcommand{\vb}{\mathrm{vb}}
\newcommand{\gf}{{\operatorname{fin}}}
\newcommand{\id}{\operatorname{id}}
\renewcommand{\int}{\operatorname{int}}

\begin{document}

\title{Gromov's Oka principle for equivariant maps}

\author{Frank Kutzschebauch, Finnur L\'arusson, Gerald W.~Schwarz}

\address{Frank Kutzschebauch, Institute of Mathematics, University of Bern, Sidlerstrasse 5, CH-3012 Bern, Switzerland}
\email{frank.kutzschebauch@math.unibe.ch}

\address{Finnur L\'arusson, School of Mathematical Sciences, University of Adelaide, Adelaide SA 5005, Australia}
\email{finnur.larusson@adelaide.edu.au}

\address{Gerald W.~Schwarz, Department of Mathematics, Brandeis University, Waltham MA 02454-9110, USA}
\email{schwarz@brandeis.edu}

\thanks{F.~Kutzschebauch was supported by Schweizerischer Nationalfonds grant 200021-178730.  F.~L\'arusson was supported by Australian Research Council grant DP150103442.}


\date{16 Dec 2019.  Minor edits 9 Aug 2020.  Erratum to Lemma 5.4 added 29 Sep 2023}


\begin{abstract} 
We take the first step in the development of an equivariant version of modern, Gromov-style Oka theory.  We define equivariant versions of the standard Oka property, ellipticity, and homotopy Runge property of complex manifolds, show that they satisfy all the expected basic properties, and present examples.  Our main theorem is an equivariant Oka principle saying that if a finite group $G$ acts on a Stein manifold $X$ and another manifold $Y$ in such a way that $Y$ is $G$-Oka, then every $G$-equivariant continuous map $X\to Y$ can be deformed, through such maps, to a $G$-equivariant holomorphic map.  Approximation on a $G$-invariant holomorphically convex compact subset of $X$ and jet interpolation along a $G$-invariant subvariety of $X$ can be built into the theorem.  We conjecture that the theorem holds for actions of arbitrary reductive complex Lie groups and prove partial results to this effect.
\end{abstract}

\maketitle

\section{Introduction} 
\label{sec:intro}

\noindent
Oka theory is concerned with approximation and interpolation problems for holomorphic maps from Stein manifolds into target manifolds for which the obstructions to solving such problems are purely topological.  The prototypical examples of such target manifolds are homogeneous spaces of complex Lie groups.  Oka theory for homogeneous spaces was initiated by Grauert in the late 1950s and developed by him and others into the 1970s.  Modern Oka theory was launched by Gromov in a seminal paper of 1989 \cite{Gromov1989}.  He discovered a linearisation method that has made it possible to extend the theory well beyond the class of homogeneous spaces.  The most natural larger class of target manifolds emerged later: it is the class of Oka manifolds.  Since 2000, the foundations of Oka theory and a wide range of applications have been so actively developed that Oka theory can now be considered a subfield of complex analysis and complex geometry in its own right.  The definitive reference on Oka theory is the monograph \cite{Forstneric2017}; see also the survey \cite{FL2011}.

It is always of interest and often useful to adapt a mathematical theory to the presence of symmetries.  Classical Oka theory and geometric invariant theory were first brought together in the work of Heinzner and Kutzschebauch \cite{HK1995}.  The present authors have continued this development, most recently with an equivariant parametric Oka principle for bundles of homogeneous spaces \cite{KLS2018}.  The present paper is the first step towards an equivariant version of modern, Gromov-style Oka theory.

The most basic result of Oka theory says that if $X$ is a Stein manifold and $Y$ is an Oka manifold, then every continuous map $f:X\to Y$ can be deformed to a holomorphic map.  In \cite{KLS2018}, we proved that if $f$ is equivariant with respect to holomorphic actions of a reductive complex Lie group $G$ on $X$ and $Y$, then $f$ can be deformed through equivariant maps to an equivariant holomorphic map -- provided that the $G$-action on $Y$ factors through a transitive action of some other complex Lie group (not necessarily reductive) on $Y$, so $Y$ is in particular homogeneous.  The main result of \cite{KLS2018} says much more, but the homogeneity assumption is essential.  Our goal here, in the spirit of Gromov, is to remove it.

In Section \ref{s:equivar-Oka}, we define a complex manifold $Y$ on which a reductive complex Lie group $G$ acts by biholomorphisms to be Oka with respect to the action, or simply $G$-Oka, if, for every reductive closed subgroup $H$ of $G$, the submanifold $Y^H$ of points fixed by $H$ is Oka in the usual sense.  The basic examples of $G$-Oka manifolds are $G$-modules\footnote{A $G$-module is a finite-dimensional vector space with an action of $G$ by linear maps.} and $G$-homogeneous spaces.  There are also other examples.   The various results that we prove about the equivariant Oka property, from easy propositions to the main theorems of the paper, suggest that the definition is the right one.

Gromov's notion of ellipticity is the primary geometric sufficient condition for a manifold to be Oka.  In Section \ref{s:equivar-ellipticity}, we present a natural equivariant version of ellipticity and show that it implies the equivariant Oka property.  We prove basic results about equivariant ellipticity, in particular its behaviour upon passing to a covering space, and present some examples.

In Sections \ref{s:finite-action} and \ref{s:reductive-action}, we prove the main results of the paper.  They may be summarised as follows.

\begin{theorem*}
Let $G$ be a reductive complex Lie group and let $K$ be a maximal compact subgroup of $G$.  Let $X$ be a Stein $G$-manifold and $Y$ a $G$-Oka manifold.  Suppose that one of the following conditions holds.
\begin{enumerate}
\item  $G$ is finite (so $G=K$).
\item  All the stabilisers of the $G$-action on $X$ are finite.
\item  $X$ has a single slice type.
\end{enumerate}
Then every $K$-equivariant continuous map $f:X\to Y$ is homotopic, through such maps, to a $G$-equivariant holomorphic map.  

Suppose that {\rm (1)} or {\rm (2)} hold.  If $f$ is holomorphic on a neighbourhood of a $G$-invariant subvariety $Z$ of $X$ and on a neighbourhood of a $K$-invariant $\O(X)$-convex compact subset $A$ of $X$, and $\ell\geq 0$ is an integer, then the homotopy can be chosen so that the intermediate maps agree with $f$ to order $\ell$ along $Z$ and are uniformly close to $f$ on $A$.
\end{theorem*}

Of course condition (1) implies condition (2), but {\rm (1)} is important enough to be stated separately.  We prove the theorem assuming (1) in Section \ref{s:finite-action} and assuming (2) or (3) in Section \ref{s:reductive-action}.  The property ascribed to the manifold $Y$ in the theorem is called the equivariant basic Oka property with approximation and jet interpolation (abbreviated $G$-BOPAJI).  Thus the theorem says that if $G$ is finite and $Y$ is $G$-Oka, then $Y$ satisfies $G$-BOPAJI.  A key ingredient in the proof of the theorem is Forstneri\v c's Oka principle for sections of branched holomorphic maps (\cite[Theorem 2.1]{Forstneric2003}; see also \cite[Theorem 6.14.6]{Forstneric2017}), the only available Oka principle in modern Oka theory that does not require the map in question to be a submersion.

We conjecture, and hope to prove in future work without any further assumptions, that when a reductive group $G$ acts on a manifold $Y$ in such a way that $Y$ is $G$-Oka, then $Y$ satisfies $G$-BOPAJI (and even the stronger, parametric version of $G$-BOPAJI).  

One of the remarkable features of the standard Oka property is that it has many nontrivially equivalent formulations.  We show that for a finite group $G$, the $G$-Oka property, $G$-BOPAJI, and the weaker properties $G$-BOPA (approximation without interpolation) and $G$-BOPI (plain interpolation without approximation) are equivalent.  We also generalise the theorem that an Oka Stein manifold is elliptic to an equivariant setting.

In Section \ref{s:equivar-Runge}, we introduce the equivariant version of the so-called homotopy Runge property from standard Oka theory and show that it has all the expected basic properties, in particular it implies the $G$-Oka property.  It is easily seen that a $G$-module is $G$-Runge, and $G$ itself is $G$-Runge, but (somewhat surprisingly) we have not been able to find a simple proof that a $G$-homogeneous space is $G$-Runge.  In Section \ref{s:ellipticity-implies-Runge}, we prove that if $G$ is a reductive group and $Y$ is $G$-elliptic, then $Y$ is $G$-Runge.  It follows that a $G$-homogeneous space is $G$-Runge.  The proof is an adaptation of Gromov's linearisation method to an equivariant setting.  We conjecture that the hypothesis of $G$-ellipticity can be replaced by the $G$-Oka property.  For now we have proved that for a finite group $G$ and a Stein $G$-manifold, the $G$-Oka property, the $G$-Runge property, and $G$-ellipticity are equivalent.

\section{Equivariant Oka property}  \label{s:equivar-Oka}

\noindent
Let a compact real Lie group $K$ act on a complex manifold $Y$.  We always mean a continuous and hence real analytic action by biholomorphisms.  We say that $Y$ is $K$-\emph{Oka}, or more precisely, Oka with respect to the given $K$-action, if the fixed-point manifold $Y^L$ is Oka for every closed subgroup $L$ of $K$.  (By Bochner's linearisation theorem (\cite{Bochner1945}; see also \cite{Akhiezer1995}), the subvariety $Y^L$ is indeed smooth.  If $L$ is not contained in the stabiliser of any point in $Y$, then $Y^L$ is empty and by default Oka.  Of course $Y^L$ need not be connected.)

If a reductive complex Lie group $G$ acts on $Y$ (holomorphically by biholomorphisms), then we say that $Y$ is $G$-\emph{Oka} if it is $K$-Oka for some, or equivalently all, maximal compact subgroups $K$ of $G$.  Equivalently, $Y^H$ is Oka for all reductive closed subgroups $H$ of $G$.

Here are the basic properties of the equivariant Oka property.

\begin{proposition}  \label{p:Oka-properties}
Let a compact real Lie group $K$ act on a complex manifold $Y$.
\begin{enumerate}
\item  If $K$ acts trivially on $Y$, then $Y$ is $K$-Oka if and only if $Y$ is Oka.
\item  If $Y$ is $K$-Oka and $L$ is a closed subgroup of $K$, then $Y$ is $L$-Oka with respect to the restriction of the action to $L$.
\item  If $Y$ is $K$-Oka, then $Y$ is Oka.
\item  If $Y_j$ is $K_j$-Oka, $j=1,2$, then $Y_1\times Y_2$ is $K_1\times K_2$-Oka.  
\item  If $Y_1$ and $Y_2$ are $K$-Oka, then $Y_1\times Y_2$ is $K$-Oka with respect to the diagonal action.
\item  A holomorphic $K$-retract of a $K$-Oka manifold is $K$-Oka.
\item  If $Y$ is the increasing union of $K$-Oka $K$-invariant domains, then $Y$ is $K$-Oka.
\end{enumerate}
\end{proposition}

\begin{proof}
These properties follow easily straight from the definition or from basic properties of Oka manifolds.  For (4), note that stabilisers of the action of $K_1\times K_2$ on $Y_1\times Y_2$ are of the form $L_1\times L_2$, where $L_j$ is a stabiliser of the action of $K_j$ on $Y_j$, $j=1,2$, and $(Y_1\times Y_2)^{L_1\times L_2}=Y_1^{L_1}\times Y_2^{L_2}$.  For (6), note that if $Z\hookrightarrow Y\to Z$ is a $K$-retract and $L$ is a closed subgroup of $K$, then $Z^L$ is a retract of $Y^L$.
\end{proof}

Here are the basic examples of equivariantly Oka manifolds.  (The proposition may also be deduced from Propositions \ref{p:elliptic-implies-Oka} and \ref{p:elliptic-examples}.)

\begin{proposition}  \label{p:Oka-examples}
Let $G$ be a reductive complex Lie group.
\begin{enumerate}
\item  A $G$-module is $G$-Oka.
\item  A $G$-homogeneous space is $G$-Oka.
\end{enumerate}
\end{proposition}

Note that a Riemann surface that is Oka ($\C$, $\C^*$, an elliptic curve, or the Riemann sphere) is Oka with respect to any action, because the fixed-point set of any subgroup is either the whole surface or discrete and hence Oka.

\begin{proof}
(1)  The fixed-point manifold of a $G$-module is a linear subspace.

(2)  Consider a homogeneous space $G/H$ and let $L$ be a reductive closed subgroup of $G$.  The following argument is due to Luna \cite[proof of Corollary 3.1]{Luna1975}.  We have an action by left multiplication of $N_G(L)$ on $(G/H)^L$, which is a submanifold of $G/H$.  The tangent space of $(G/H)^L$ at any point $x$ can be identified with $(\mathfrak g/\mathfrak h)^L$.  Since $L$ is reductive, $(\mathfrak g/\mathfrak h)^L$ is isomorphic to $\mathfrak g^L/\mathfrak h^L$.  But the Lie algebra of $N_G(L)$ is $\mathfrak g^L$.  Hence the $N_G(L)$-orbit of $x$ is open in $(G/H)^L$, so it is open in the component (connected or irreducible: they are the same) of $(G/H)^L$ through $x$, so it contains this component.  Since $(G/H)^L$ has finitely many components, the number of $N_G(L)$-orbits in $(G/H)^L$ is finite.  Thus $(G/H)^L$ is a disjoint union of homogeneous spaces and is therefore Oka. 
\end{proof}

The following result shows that equivariantly Oka manifolds can be constructed by induction.

\begin{proposition}  \label{p:Oka-induction}
Let $G$ be a reductive complex Lie group, $H$ a reductive closed subgroup of $G$, and $Y$ an $H$-manifold.  Then   $G\times^H Y$ with its natural $G$-action is $G$-Oka if and only if $Y$ is $H$-Oka.
\end{proposition}

Recall that the twisted product $\ind_H^G Y=G\times^H Y$ is defined as the geometric quotient of $G\times Y$ by the $H$-action $h\cdot(g,y)=(gh^{-1}, hy)$.  The quotient is smooth because $H$ acts freely on $G$ and hence on $G\times Y$.   We denote the image of $(g,y)$ in $G\times^HY$ by $[g,y]$. The $G$-action on $G\times^H Y$ is induced by the $G$-action $g'\cdot(g,y) = (g'g,y)$ on $G\times Y$.  The induction functor from $H$-spaces to $G$-spaces is left adjoint to the restriction functor, which also preserves the equivariant Oka property, as noted in Proposition  \ref{p:Oka-properties}.  The manifold $G\times^H Y$ is the total space of a $G$-fibre bundle over $G/H$ with fibre $Y$, associated to the principal bundle $G\to G/H$.  Looking at the bundle $G\times Y\to G\times^H Y$ with fibre $H$ or the bundle $G\times^H Y\to G/H$ with fibre $Y$, we see that if $Y$ is Oka, then so is $G\times^H Y$, and conversely.   (For the theorem that if $E\to B$ is a holomorphic fibre bundle with Oka fibre, then $E$ is Oka if and only if $B$ is Oka, see \cite[Theorem 5.6.5]{Forstneric2017}.  For more details on twisted products in the topological setting, see \cite[Chapter II]{Bredon1972}.)

\begin{proof}
Write $X=G\times^H Y$ and let $L$ be a reductive closed subgroup of $G$. Then $X^L$ is not empty if and only if $L$ is conjugate to a subgroup of $H$, so we may assume that $L\subset H$.   Note that the preimage of $X^L$ in $G\times Y$ is
\[ \{ (g,y)\in G\times Y: g^{-1}Lg\subset H, y\in Y^{g^{-1}Lg} \}. \]
The image of $X^L$ by the projection $\pi:X\to G/H$ lies in $(G/H)^L$, which, by the proof of Proposition \ref{p:Oka-examples}, is a finite union of $N_G(L)$-orbits.  Over each such orbit, say the orbit $O$ of $gH$, the map
\[ N_G(L)\times Y^{g^{-1}Lg} \to X^L\cap \pi^{-1}(O), \quad (n,y)\mapsto [ng,y], \]
is the quotient map of the free action of the group $M=N_G(L)\cap gHg^{-1}$ on the source given by $m\cdot(n,y)=(nm^{-1},g^{-1}mgy)$.  Since the fibre is Oka, the target is Oka if and only if $Y^{g\inv L g}$ is Oka, where $g\inv Lg\subset H$ \cite[Theorem 5.6.5]{Forstneric2017}. Since $X^L$ is a finite disjoint union of such quotients, we see that $X$ is $G$-Oka   if and only if $Y$ is $H$-Oka.
\end{proof}

One of the most important ways to obtain new Oka manifolds from old uses fibre bundles with Oka fibres.  We present an equivariant version.  Let $Y\overset\pi\to Z$ be a holomorphic fibre bundle with fibre $F$.  Assume that $G$ is a reductive complex Lie group and $Y$, $Z$, and $F$ are $G$-manifolds such that $\pi$ is $G$-equivariant.  Further assume that $Z$ is Stein.

\begin{lemma}   \label{l:equivar-locally-trivial}
The Stein $G$-manifold $Z$ has an open cover by $G$-saturated\footnote{A set is \textit{saturated} if it is the union of fibres of the categorical quotient map $Z\to Z\git G$.} open sets $U$ such that $\pi\inv(U)$ is $G$-biholomorphic over $U$ to $U\times F$ with the diagonal $G$-action.
\end{lemma} 

\begin{proof}
Let $Gz$ be a closed orbit in $Z$.  By the holomorphic slice theorem, a $G$-saturated neighbourhood $U$ of $Gz$ is $G$-biholomorphic to $G\times^H S$, where $H=G_z$ and $S$ is a locally closed $H$-stable\footnote{For subsets of a set with a group action, we use the terms \textit{invariant} and \textit{stable} interchangeably.} submanifold of $Z$ containing $z$.  By \cite[Corollary 3, p.~332]{HK1995}, if $S$ is sufficiently small, then $\pi\inv(U)$ is $G$-biholomorphic to $G\times ^H(S\times F)$, where $H$ acts diagonally on $S\times F$.  Since $F$ is a $G$-manifold, $G\times ^H(S\times F)$ is $G$-biholomorphic to $(G\times^HS)\times F=U\times F$, proving the lemma.
\end{proof}

\begin{corollary}   \label{c:up-and-down-a-fibre-bundle}
Suppose that $F$ is $G$-Oka.  Then $Y$ is $G$-Oka if and only if $Z$ is $G$-Oka.
\end{corollary}

\begin{proof}
Let $L$ be a reductive closed subgroup of $G$.  By the lemma, $Y^L\to Z^L$ is a fibre bundle with fibre $F^L$.  By hypothesis, $F^L$ is Oka, so $Y^L$ is Oka if and only if $Z^L$ is Oka \cite[Theorem 5.6.5]{Forstneric2017}.
\end{proof}

\begin{example}   \label{e:Hirzebruch}
Let $Y_n$ be the $n^\textrm{th}$ Hirzebruch surface, $n\geq 1$ (we leave $Y_0=\mathbb P_1 \times \mathbb P_1$ out of consideration).  It is the total space of a holomorphic $\mathbb P_1$-bundle over $\mathbb P_1$ and is therefore Oka.  The reductive group $G=\GL(2,\C)$ acts on $Y_n$ in such a way that the projection $\pi:Y_n\to\mathbb P_1$ is equivariant with respect to the usual action on $\mathbb P_1$.  In fact, the quotient of $G$ by the subgroup $M$ of scalar matrices $\begin{bmatrix} \mu & 0 \\ 0 & \mu \end{bmatrix}$ with $\mu^n=1$ is a maximal reductive subgroup of the automorphism group of $Y_n$, which is isomorphic to $\C^{n+1}\rtimes G/M$ .  The automorphism group has two orbits, so $Y_n$ is not homogeneous.  For a detailed exposition of the background facts cited here, see \cite[Section 6]{Blanc2012}.  

We will show that $Y_n$ is $G$-Oka (or equivalently $G/M$-Oka).  Although $\mathbb P_1$ is $G$-Oka by Proposition \ref{p:Oka-examples}(b), Corollary \ref{c:up-and-down-a-fibre-bundle} does not help.  Indeed, the Stein hypothesis of Lemma \ref{l:equivar-locally-trivial} obviously fails, and so does the conclusion of the lemma: the only nonempty $G$-invariant open subset of $\mathbb P_1$ is $\mathbb P_1$ itself, and $Y_n$ is not biholomorphic to $\mathbb P_1\times \mathbb P_1$.

Following \cite[Section 6]{Blanc2012}, we describe $Y_n$ as
\[ Y_n = \big\{([x,y,z],[u,v]) \in \mathbb P_2\times\mathbb P_1: yv^n=zu^n \big\}, \]
with $\pi$ being the projection onto the second factor, and let $\begin{bmatrix} a & b \\ c & d \end{bmatrix} \in G$ act by sending $([x,y,z],[u,v])$ to
\[ ( [xu^n, y(au+bv)^n, y(cu+dv)^n], [au+bv, cu+dv]) \quad \textrm{if }u\neq 0, \]
and 
\[ ( [xv^n, z(au+bv)^n, z(cu+dv)^n], [au+bv, cu+dv]) \quad \textrm{if }v\neq 0. \]

Now let $H$ be a reductive closed subgroup of $G$ not consisting entirely of scalar matrices $\begin{bmatrix} \mu & 0 \\ 0 & \mu \end{bmatrix}$ with $\mu^n=1$.  We will show that the submanifold $Y_n^H$ of $Y_n$ is Oka.  It lies over $\mathbb P_1^H$ (but we cannot apply Lemma \ref{l:equivar-locally-trivial} to assert that it is a bundle over $\mathbb P_1^H$).  Over each point of $\mathbb P_1^H$, $Y_n^H$ consists of the whole $\pi$-fibre or a finite subset of it.  Hence, if $\mathbb P_1^H$ is finite, $Y_n^H$ is Oka.  Otherwise, $\mathbb P_1^H=\mathbb P_1$, so $H$ consists of scalar matrices.  Then $H$ has precisely two fixed points in each $\pi$-fibre, so $Y_n^H$ is biholomorphic to the disjoint union of two copies of $\mathbb P_1$ and is therefore Oka.
\end{example}

\begin{example}   \label{e:Oka-but-not-G-Oka}
This example shows that a $G$-manifold $Y$ can be Oka without being $G$-Oka, even for the simplest nontrivial group $G=\mathbb Z_2$.

Let $f\in\O(\C^n)$, $n\geq 2$, be a polynomial function such that $df$ vanishes nowhere on $f^{-1}(0)$.  By \cite{KK2008}, the affine algebraic manifold $X=\{(u,v,z)\in\C^{n+2}:uv=f(z)\}$ has the algebraic density property and is therefore elliptic and hence Oka.  The fixed point set $W\subset X$ of the involution $u \leftrightarrow v$ of $X$ is smooth and given by the formula $u^2=f(z)$, and $W$ is a double branched covering of $\C^n$ with branch locus $f^{-1}(0)$.  We can choose $f$ so that $f^{-1}(0)$ is not Oka.  For example, if we take $f(z_1,\ldots,z_n)=z_1(z_1-1)z_2-1$, then $f^{-1}(0)$ is isomorphic to $\C\setminus\{0,1\}\times \C^{n-2}$.

We do not know whether $W$ is Oka in general.  If it is, then our promised example is $Y=W$ and $Y^G=f^{-1}(0)$.  If it is not, then the example is $Y=X$ and $Y^G=W$.
\end{example}

\section{Equivariant ellipticity}  \label{s:equivar-ellipticity}

\noindent
A \emph{spray} on a manifold $Y$ is a holomorphic map $s:E\to Y$ defined on the total space of a holomorphic vector bundle $E$ over $Y$ such that $s(0_y)=y$ for all $y\in Y$.  The spray is said to be \emph{dominating at} $y\in Y$ if $s\vert E_y\to Y$ is a submersion at $0_y$.  The spray is said to be \emph{dominating} if it is dominating at every point of $Y$.   Finally, $Y$ is said to be \emph{elliptic} if it admits a dominating spray.

Now suppose that a complex Lie group $G$ acts on $Y$.  We say that $s$ is a $G$-\emph{spray} if the $G$-action on $Y$ lifts to an action of $G$ on $E$ by vector bundle isomorphisms such that both $s$ and the projection $E\to Y$ are equivariant.  We say that $Y$ is $G$-\emph{elliptic} if it admits a dominating $G$-spray.  (Similarly we can define the notion of $G$-subellipticity, but we will not consider it in this paper.)

Here are the basic properties of $G$-ellipticity.  Note the similarity to Proposition \ref{p:Oka-properties}.

\begin{proposition}  \label{p:elliptic-properties}
Let a complex Lie group $G$ act on a complex manifold $Y$.
\begin{enumerate}
\item  If $G$ acts trivially on $Y$, then $Y$ is $G$-elliptic if and only if $Y$ is elliptic.
\item  If $Y$ is $G$-elliptic, then $Y$ is $H$-elliptic for every subgroup $H$ of $G$.
\item  If $Y$ is $G$-elliptic, then $Y$ is elliptic and hence Oka.
\item  If $Y_k$ is $G_k$-elliptic, $k=1,2$, then $Y_1\times Y_2$ is $G_1\times G_2$-elliptic.  
\item  If $Y_1$ and $Y_2$ are $G$-elliptic, then $Y_1\times Y_2$ is $G$-elliptic with respect to the diagonal action.
\item  A holomorphic $G$-retract of a $G$-elliptic manifold is $G$-elliptic.
\end{enumerate}
\end{proposition}

\begin{proof}
(1)--(5) are evident.

(6)  Let $Z$ be a $G$-retract of a $G$-elliptic manifold $Y$ with a dominating $G$-spray $s$ defined on a bundle $E$ over $Y$.  To see that $Z$ is $G$-elliptic, postcompose the restriction of $s$ to $E\vert Z$ by a $G$-retraction of $Y$ onto $Z$.
\end{proof}

Ellipticity is the primary geometric sufficient condition for a manifold to be Oka.  The next result shows that this still holds in the presence of a reductive group action.

\begin{proposition}   \label{p:elliptic-implies-Oka}
Let $G$ be a reductive complex Lie group.  If $Y$ is a $G$-elliptic manifold, then $Y$ is $G$-Oka.
\end{proposition}

\begin{proof}
Let $E\to Y$ be a holomorphic $G$-vector bundle with a dominating $G$-spray $s$ and let $L$ be a reductive subgroup of $G$.  Consider the $L$-vector bundle $E_0:=E\vert Y^L$ over $Y^L$.  By the rigidity of representations of reductive groups, over each connected component of $Y^L$, the $L$-actions on the fibres of $E$ are mutually isomorphic; in particular, the subspaces of fixed points have the same dimension.  Hence, $E_0^L$ is a holomorphic vector bundle over $Y^L$ with a dominating spray map induced by $s$, so $Y^L$ is elliptic and therefore Oka.
\end{proof}

Here are the basic examples of $G$-elliptic manifolds.

\begin{proposition}  \label{p:elliptic-examples}
\begin{enumerate}
\item  A $G$-module is $G$-elliptic.
\item  A $G$-homogeneous space is $G$-elliptic.
\end{enumerate}
\end{proposition}

\begin{proof}
(1)  Let $V$ be a $G$-module and $s:V\times V\to V$, $(z,v)\mapsto z+v$, be the obvious dominating spray defined on the trivial bundle over $V$ with fibre $V$.  The $G$-action on $V$ lifts to the diagonal action on $V\times V$ which makes both $s$ and the projection equivariant.

(2)  Let $Y$ be $G$-homogeneous.  Let $Y\times\mathfrak g\to Y$ be the trivial $G$-vector bundle where $G$ acts on $\mathfrak g$ by the adjoint representation. Let $s$ be the dominating spray $Y\times\mathfrak g\to Y$, $(y,v)\mapsto \exp(v)\cdot y$.   Then both $s$ and the projection are equivariant with respect to the $G$-action on $Y\times\mathfrak g$.
\end{proof}

There is a more elementary dominating equivariant spray on a complex Lie group $G$ acting on itself by left multiplication, namely $G\times\mathfrak g\to G$, $(g,v)\mapsto g\exp(v)$, where $G$ acts trivially on $\mathfrak g$.

The following result shows that $G$-elliptic manifolds can be constructed by induction.

\begin{proposition}  \label{p:elliptic-induction}
Let $G$ be a complex Lie group, $H$ a closed complex subgroup of $G$, and $Y$ an $H$-elliptic manifold.  Then the manifold $G\times^H Y$ with its natural $G$-action is $G$-elliptic.
\end{proposition}

\begin{proof}
Let  $\pi: E\to Y$ be a holomorphic $H$-vector bundle on $Y$ together with an $H$-equivariant holomorphic dominating spray $t: E\to Y$. Let $F=G\times^HE$, which is a $G$-vector bundle over $X=G\times^H Y$ with projection $F\to X$, $[g,v]\mapsto[g,\pi(v)]$.  The induced spray map $\pi^*t : F\to X$, $[g,v]\mapsto [g,t(v)]$, is $G$-equivariant and dominating on the fibres of $X\to G/H$.  Let $X_{\mathfrak g}$ denote  the trivial $G$-vector bundle $X\times\mathfrak g$. Here $G$ acts on $\mathfrak g$ by the adjoint representation.  We have the spray map 
\[ s: X_{\mathfrak g}\to X,\quad ([g,y],A)\mapsto [\exp(A)g,y]. \]
The composition $s: X_{\mathfrak g}\to X\to G/H$ is dominating and 
\[ (\pi^*t,s): F\oplus X_{\mathfrak g}\to X \]
is dominating and $G$-equivariant.
\end{proof}

\begin{example}
The first examples of elliptic manifolds that are not necessarily homogeneous were complements in $\C^n$ of algebraic subvarieties of codimension at least 2 \cite[Paragraph 0.5.B(iii)]{Gromov1989}.  Complements of tame analytic subvarieties of codimension at least 2 in $\C^n$ are also elliptic \cite[Proposition 6.4.1]{Forstneric2017}.  Here is a simple equivariant example.

Let $\Z_2$ act on $\C^2$ by reflection in the origin and let $X=\C^2\setminus(\Z\times\{0\})$.  Consider vector fields on $X$ of the form
\[ v(z,w) = f(az+bw)\bigg(b\frac\partial{\partial z} - a\frac\partial{\partial w}\bigg), \]
with $a,b\in\C$, $a\neq 0$, $f\in\O(\C)$ odd and zero on $a\Z$.  The field $v$ is complete and its flow is given by the formula
\[ \phi_v^t(z,w) = (z,w)+tf(az+bw)(b, -a), \quad t\in\C. \]
Finitely many fields $v_1,\ldots,v_k$ suffice to span the tangent space of $X$ at each point.  Then a dominating $\Z_2$-equivariant spray on $X$, defined on the trivial bundle of rank $k$ with the trivial $\Z_2$-action, is given by the formula
\[ s(z,w,t_1,\ldots,t_k) = (\phi_{v_k}^{t_k}\circ\cdots\circ\phi_{v_1}^{t_1})(z,w), \]
so $X$ is $\Z_2$-elliptic.  It was pointed out already in \cite[Satz 4.13]{Kaup1967} that $X$ is not homogeneous with respect to any Lie group, real or complex, although the group of holomorphic automorphisms of $X$ acts transitively on $X$.
\end{example}

\begin{example}
Example \ref{e:Oka-but-not-G-Oka} shows that an elliptic manifold with a $G$-action need not be $G$-elliptic, even for $G=\mathbb Z_2$.
\end{example}

It is well known and easily seen that any covering space (unbranched; finite or infinite) of an elliptic manifold is elliptic.  For equivariant ellipticity, the rather subtle issue of lifting group actions to covering spaces arises.  In the following, we have relied on Bredon's account in \cite[Section I.9]{Bredon1972}.

Let $Y$ be a $G$-elliptic manifold where $G$ is reductive.  We take $Y$ to be connected, but the group $G$ need not be connected.  We may assume that $G$ acts effectively on $Y$. By definition, there is a $G$-vector bundle $E\to Y$ and a dominating equivariant spray $s: E\to Y$.

First assume that $G$ is semisimple (in particular connected).  Then the universal covering $G'$ of $G$ is a finite covering and $G'$ is again semisimple.  Let $Y'$ be a covering space of $Y$. Then the action of $G$ on $Y$ lifts to an action of $G'$ on $Y'$ (Bredon's lift is purely topological, but since it covers a holomorphic action on $Y$, the lifted action is holomorphic).  Replacing $G'$ by a finite quotient, we can assume that $G'$ acts effectively on $Y'$.  If $\pi: Y'\to Y$ is the covering map, then $E'=\pi^*E\to Y'$ is a holomorphic vector bundle on $Y'$ with a holomorphic action of $G'$, which is the lift of the $G$-action on $E$.  Thus $E'\to Y'$ is a holomorphic $G'$-vector bundle.  The spray map $s:E\to Y$ lifts to the spray map $\pi^*s : E'\to Y'$, which is also dominating since it is dominating below.  We then have the following result.

\begin{proposition}
Let $G$ be semisimple and let $Y$ be connected and $G$-elliptic.  Let $Y'$ be a covering space of $Y$.  Then $Y'$ is $G'$-elliptic for a finite covering $G'$ of $G$.
\end{proposition}

Here and in the two propositions that follow, $G'$ is uniquely determined as the quotient of the universal covering group of $G$ that acts effectively on $Y'$.

Now suppose that $G$ is not semisimple but only reductive, and connected.  Then the universal covering of $G$ is not reductive, so we cannot hope to have an action of a reductive group $G'$ covering $G$ on an infinite covering of $Y$.  So let $Y'$ be a finite covering of $Y$.  Then the universal covering $G'$ of $G$ acts holomorphically on $Y'$, but not effectively.  The kernel of the action is a normal subgroup $H'$ of the centre of $G'$ such that $G''=G'/H'$ is a finite covering of $G$ and thus reductive.  It might happen that $Y'\to Y$ is not finite, yet a finite covering of $G$ still acts on $Y'$, covering the action of $G$ on $Y$.  Then the argument above shows the following.

\begin{proposition}
Let $G$ be reductive and connected and let $Y'$ be a covering of the connected $G$-elliptic manifold $Y$, such that a finite covering $G'$ of $G$ acts on $Y'$ covering the action of $G$ on $Y$.  Then $Y'$ is $G'$-elliptic.  In particular, this holds if $Y'\to Y$ is a finite covering.
\end{proposition}

If we do not care about $G'$ being reductive, then we can take $Y'$ to be any covering space of $Y$.

Let us now drop the assumption that $G$ is connected.  This complicates matters.  In \cite[Section I.9]{Bredon1972}, Bredon first considers the case where $Y^G$ is not empty.  Let $y_0\in Y^G$ and let $Y'$ be a covering of $Y$ with $y_0'$ covering $y_0$.  Then there is an action of $G$ itself on $Y'$ if and only if the image of the induced map $\pi_1(Y',y_0')\to\pi_1(Y,y_0)$ is invariant under the action of $G$.  Note the this is only a question about the action of $G/G^0$ since the identity component $G^0$ acts trivially on $\pi_1(Y,y_0)$.  In any case, if the group action lifts to $Y'$, then $Y'$ is clearly $G$-elliptic as above.

Now suppose that $Y$ has no $G$-fixed points.  Let $Y'\to Y$ be a Galois (also called normal) 
covering of $Y$, for example the universal covering.  Then there is an action of a covering group of $G'$ on $Y'$ if and only if the following topological condition holds.
\begin{equation}   \label{eqn}
\textrm{The action of } G \textrm{ preserves the image of } [S^1,Y'] \textrm{ in }[S^1,Y].
\end{equation}
Here, $[S^1,X]$ denotes homotopy classes of continuous maps from the circle $S^1$ to $X$.  Note that if $G$ is connected, \eqref{eqn} is automatically satisfied.

\begin{proposition}
Let $G$ be reductive and let $Y'$ be a Galois covering of the connected $G$-elliptic manifold $Y$ satisfying \eqref{eqn}. Then there is an effective  holomorphic action of a covering group $G'$ of $G$ on $Y'$ covering the action on $Y$, and $Y'$ is $G'$-elliptic. If the covering $G'\to G$ is finite, then $G'$ is reductive.
\end{proposition}

\begin{remark}
Note that the proposition holds when $G$ is the trivial group. Then $G'$ is the group of deck transformations of $Y'$ and $Y'$ is $G'$-elliptic.  Conversely, suppose that $Y'$ is $G'$-elliptic with holomorphic $G'$-vector bundle $E'$ and dominating spray map $s: E'\to Y'$.  Since $G'$ acts freely on $Y'$, $E'/G'$ is a holomorphic vector bundle on $Y$ and $s$ induces a dominating spray map $s/G: E'/G\to Y'/G=Y$.  Thus $Y'$ is $G'$-elliptic if and only if $Y$ is elliptic. It is unknown whether ellipticity of $Y'$ implies that $Y$ is elliptic.
\end{remark}

\section{Equivariant Oka principle for finite group actions}   \label{s:finite-action}

\noindent
In this section and the next one, we consider equivariant Oka principles for $G$-equivariant maps from a Stein $G$-manifold $X$ to a $G$-Oka manifold $Y$, where $G$ is a reductive complex Lie group.  Since $G$ has finitely many connected components, we may write $X=\bigcup_n X_n$ where each $X_n$ is a union of components of $X$ that $G$ permutes transitively.  Thus we can always reduce to the case that $G$ acts transitively on the set of components of $X$. Pick a component $X'$ and let $G'$ denote the stabiliser of $X'$ in $G$.  Then $G'\supset G^0$ is a normal subgroup of $G$ and  $X\simeq G\times^{G'}X'$.  A $G$-equivariant map $f:X\to Y$ is determined uniquely by the restriction $f':X'\to Y$ which is $G'$-equivariant.  Thus we may reduce to the case that $X$ is connected.  The same reduction works when considering $K$-equivariant maps from $X$ to $Y$, where $K$ is a maximal compact subgroup of $G$.

Here is our equivariant Oka principle for finite group actions, with approximation and jet interpolation.

\begin{theorem}   \label{t:main}
Let $G$ be a finite group.  Let $X$ be a  Stein $G$-manifold and let $Y$ be a $G$-Oka manifold.  Then every $G$-equivariant continuous map $f:X\to Y$ is homotopic, through such maps, to a $G$-equivariant holomorphic map.

{\rm (a)}  If $f$ is holomorphic on a $G$-invariant subvariety $Z$ of $X$, then the homotopy can be chosen to be constant on $Z$.

{\rm (b)}  If $f$ is holomorphic on a neighbourhood of a $G$-invariant subvariety $Z$ of $X$ and on a neighbourhood of a $G$-invariant $\O(X)$-convex compact subset $A$ of $X$, and $\ell\geq 0$ is an integer, then the homotopy can be chosen so that the intermediate maps agree with $f$ to order $\ell$ along $Z$ and are uniformly close to $f$ on $A$.
\end{theorem}

\begin{proof} 
We may assume that $X$ is connected.  We first prove (a) and then deduce (b) from (a).  Since $G$ is finite, there are finitely many Luna strata in $Q=X/G$.  We will use the stratification by the irreducible components of the Luna strata.  If $S$ is such a stratum, then its closure is the union of $S$ and strata of strictly lower dimension.  Let $Q_i$ be the union of the strata of dimension at most $i\geq 0$.  Then we have a filtration $Q=Q_m\supset Q_{m-1}\supset\cdots\supset Q_0\supset Q_{-1}=\varnothing$ of $Q$ by subvarieties.  Let $\pi:X\to Q$ be the quotient map.  Each difference $S_k=Q_k\setminus Q_{k-1}$, $k=0,\ldots,m$, is smooth and each of its connected components is contained in a  Luna stratum.  In the following, we will consider complex subspaces of the form $\pi\inv(W)$ for subvarieties $W$ of $Q$ with their reduced structure as subvarieties.

Let $f:X\to Y$ be a continuous $G$-map, holomorphic on a $G$-invariant subvariety $Z$ of $X$.  We will show that there is a homotopy of continuous $G$-maps from $f$ to a holomorphic map, such that the homotopy is constant on $Z$.  On $\pi^{-1}(Q_0)\cup Z$, let $f_0=f$.  Note that $\pi^{-1}(Q_0)$ is discrete (possibly empty).  The proof now proceeds inductively through the following argument for $k=1,\ldots,m$.

\smallskip\noindent
\textbf{Claim 1.}  Suppose that we have a homotopy of $f\vert{\pi^{-1}(Q_{k-1})\cup Z}$, through continuous $G$-maps, constant on $Z$, to a holomorphic $G$-map $f_{k-1}:\pi^{-1}(Q_{k-1})\cup Z \to Y$.  Then $f_{k-1}$ and the homotopy extend to a $G$-invariant neighbourhood of $\pi^{-1}(Q_{k-1})\cup Z$ in $X$.

\smallskip\noindent
\textit{Proof of Claim 1.}  For $k=1$, we start with the constant homotopy.  The graph $\Gamma$ of $f_{k-1}$ in $X \times Y$ is a Stein subvariety, hence so is $\Gamma/G\subset (X\times Y)/G$.  By Siu \cite{Siu1976}, $\Gamma/G$ has a  Stein neighbourhood in $(X\times Y)/G$, so $\Gamma$ has a $G$-invariant Stein neighbourhood $V$ in $X \times Y$.  Since $G$ is finite, $V$ has only finitely many slice types, so by Heinzner's equivariant embedding theorem \cite{Heinzner1988}, $V$ embeds equivariantly as a closed $G$-invariant submanifold in a $G$-module $\C^N$ (since $X$ and $Y$ are smooth, $V$ is smooth).  Also, $V$ has a $G$-invariant tubular neighbourhood $W$ in $\C^N$ with a $G$-equivariant holomorphic retraction $W\to V$ (the standard construction can easily be made equivariant with respect to a compact group; see Proposition \ref{prop:tubularneighbourhood} below).

Let $U$ be a $G$-invariant Stein neighbourhood of $\pi^{-1}(Q_{k-1}) \cup Z$ and consider the holomorphic map from $\pi^{-1}(Q_{k-1}) \cup Z$ to $\C^N$ given by $x\mapsto (x, f_{k-1}(x))$.  By the Cartan extension theorem, it can be extended to a holomorphic map $U\to\C^N$ and made $G$-equivariant by averaging.  After further shrinking $U$ if necessary, the extension will take $U$ into $W$.  Then, postcomposing by the retraction $W\to V$ followed by the projection to $Y$, we obtain the desired extension $f_k'$ of $f_{k-1}$.

Since $Y$ is an absolute neighbourhood retract in the category of metrisable $G$-spaces \cite[Theorems 6.4 and 8.8]{Murayama1983} and by the tube lemma, the given homotopy can be extended to a homotopy of continuous $G$-maps $U\to Y$ joining $f$ and $f_k'$, possibly with a smaller $U$.  We now have a $G$-invariant neighbourhood $U$ of $\pi^{-1}(Q_{k-1})\cup Z$ in $X$ and a holomorphic $G$-map $f'_k:U \to Y$, homotopic to $f\vert U$ through continuous $G$-maps, such that the homotopy is constant on $Z$.  

\smallskip\noindent
\textbf{Claim 2.}  There is a homotopy of $f\vert \pi\inv(Q_k)\cup Z$, through continuous $G$-maps, to a holomorphic $G$-map $f_k:\pi^{-1}(Q_k)\cup Z\to Y$, such that the homotopy is constant on $Z$.

\smallskip\noindent
\textit{Proof of Claim 2.}  Let $S$ be a connected component of $S_k=Q_k\setminus Q_{k-1}$.  It is a connected component of a Luna stratum.  The (reduced) fibres of $\pi\inv(S)\to S$ are $G$-orbits with stabiliser conjugate to a fixed subgroup $H$ of $G$.  A $G$-map from $\pi\inv(S)$ to $Y$ is the same thing as a section of the bundle $(\pi\inv(S)\times Y)/G$ over $\pi\inv(S)/G=S$.  This bundle is the same as the one associated to the principal $N_G(H)/H$-bundle $\pi\inv(S)^H$ with fibre $Y^H$.  We can work simultaneously with all the components $S$ of $S_k$.   Let $E$ denote the quotient $((\pi\inv(Q_k)\cup Z)\times Y)/G$.  It is a reduced complex space over $Q_k\cup\pi(Z)$ (which is Stein) which, over $S_k$, is locally a fibre bundle with Oka fibres of the form $Y^H$.

Choose $G$-invariant neighbourhoods $U''$ and $U'$ of $\pi^{-1}(Q_{k-1})\cup Z$ in $\pi^{-1}(Q_k)\cup Z$ such that $\overline{U''}\subset U'$, $\overline{U'}\subset U$.  Modify the homotopy of $f$ and $f'_k$ so that it remains the same on  $U''$ and is constantly $f$ over the complement of $U'$.  Thus we may reduce to the case that $f$ itself is holomorphic on $U''$.  

The graph of $f$ over $\pi\inv(Q_k)\cup Z$ gives a continuous section $\sigma$ of $E$, holomorphic over $\pi(U'')$.  By a theorem of Forstneri\v c \cite[Theorem 6.14.6]{Forstneric2017}, after possibly shrinking $U''$, there is a homotopy of continuous sections $\sigma_t$ of $E$, holomorphic on $\pi(U'')$ and agreeing with $\sigma_0=\sigma$ on $\pi(Z)$, such that $\sigma_1$ is holomorphic.  The sections $\sigma_t$ give continuous $G$-maps $h_t: \pi\inv(S_k)\cup Z\to Y$, holomorphic on $U''\cap \pi\inv(S_k)$.  The maps $h_t$ induce continuous maps $\pi\inv(Q_k)\to Y/G$.  Since $Y\to Y/G$ is a finite map, by the Riemann extension theorem, the maps $h_t$ extend holomorphically across $\pi\inv(Q_{k-1})$ and the homotopy extends as well.  (We provide more details in the more general setting of the proof of Theorem \ref{t:second-main} below).  Thus we have a homotopy of $f\vert\pi\inv(Q_k)\cup Z$, constant on $Z$, to a holomorphic $G$-map $f_k:\pi\inv(Q_k)\cup Z\to Y$ as desired, and the proof of Claim 2 and thereby (a) is complete.

\smallskip

We now deduce (b) from (a) using an equivariant adaptation of the proof of \cite[Theorem 1]{Larusson2005}.  Suppose that the continuous $G$-map $f:X\to Y$ is holomorphic on a neighbourhood of a $G$-invariant subvariety $Z$ of $X$ and on a neighbourhood of a $G$-invariant $\O(X)$-convex compact subset $A$ of $X$.  Let $\ell\geq 0$ be an integer.  Take a Stein neighbourhood $U$ of $Z\cup A$ on which $f$ is holomorphic \cite[Theorem 3.2.1]{Forstneric2017}. We may assume that $U$ is $G$-invariant.  Let $\phi:U\to\C^N$ be a $G$-equivariant holomorphic embedding of $U$ as a closed $G$-invariant submanifold of a $G$-module.  

The inclusion $i:U\hookrightarrow X$ factors through the Stein manifold $M=X\times\C^N$ (endowed with the diagonal $G$-action) as $U\overset j \to M \overset \pi \to X$, where $j=(i,\phi)$ and $\pi$ is the projection onto the first factor.  The continuous $G$-map $f\circ\pi:M\to Y$ is holomorphic on the submanifold $j(U)$ of $M$.  By (a), $f\circ\pi$ can be deformed through continuous $G$-maps $h_t:M\to Y$, $t\in [0,1]$, to a holomorphic map $h_1$, with the homotopy being constant on $j(U)$.  Using the Cartan extension theorem combined with the Oka-Weil approximation theorem for coherent analytic sheaves, followed by averaging over $G$, we can approximate $\phi$ uniformly on $A$ by a holomorphic $G$-map $\psi:X\to\C^N$ which agrees with $\phi$ to order $\ell$ along $Z$.  Then the $G$-maps $f_t=h_t\circ(\id_X,\psi):X\to Y$ form a homotopy joining $f$ to the holomorphic map $f_1=h_1\circ(\id_X,\psi)$ and agree with $f$ to order $\ell$ along $Z$.  To see that the maps $f_t$ approximate $f$ on $A$, note that the maps $h_t$ are uniformly continuous on compact neighbourhoods $\Omega$ of $j(A)$, so we can make $f_t$ close to $f$ on $A$ by choosing $\Omega$ small and choosing $\psi$ such that $(a,\psi(a))\in\Omega$ for all $a\in A$.
\end{proof}

Following the standard terminology of Oka theory, we call the property ascribed to the manifold $Y$ in Theorem \ref{t:main}(a) the \emph{basic $G$-Oka property with interpolation} ($G$-BOPI).  We call the property ascribed to $Y$ in Theorem \ref{t:main}(b) the \emph{basic $G$-Oka property with approximation and jet interpolation} ($G$-BOPAJI), and when the subvariety $Z$ is empty, we call it the \emph{basic $G$-Oka property with approximation} ($G$-BOPA).

\begin{corollary}  \label{c:equivalence}
Let a finite group $G$ act on a complex manifold $Y$.  The following properties of $Y$ are equivalent: $G$-BOPA, $G$-BOPI, $G$-BOPAJI, and the $G$-Oka property.
\end{corollary}

\begin{proof}
By Theorem \ref{t:main}(b), the $G$-Oka property implies $G$-BOPAJI, which obviously implies $G$-BOPI, which in turn implies $G$-BOPA by the proof of Theorem \ref{t:main}(b).  Thus it remains to show that if $Y$ satisfies $G$-BOPA, then $Y$ is $G$-Oka.  Let $H$ be a subgroup of $G$.  If $X$ is a Stein $H$-manifold, then there is a natural homeomorphism between the space of continuous or holomorphic $G$-maps $G\times^H X \to Y$ and the space of continuous or holomorphic $H$-maps $X\to Y$, respectively.  The $G$-manifold $G\times^H X$ is Stein, and there is a natural bijection between $H$-invariant $\O(X)$-convex compact subsets of $X$ and $G$-invariant $\O(G\times^H X)$-convex compact subsets of $G\times^H X$, with $A\subset X$ corresponding to $G\times^H A\subset G\times^H X$.  It follows that $Y$ satisfies $H$-BOPA.  If $X$ is a Stein manifold with a trivial $H$-action, then an $H$-map $X\to Y$ is nothing but a map $X\to Y^H$, so $Y^H$ satisfies BOPA and is therefore Oka.
\end{proof}

If a Stein manifold is Oka, then it is elliptic (\cite[3.2.A]{Gromov1989}; see also \cite[Proposition 5.6.15]{Forstneric2017}).  Here is the equivariant version of this result.

\begin{corollary}   \label{c:Oka-implies-elliptic}
Let $G$ be a finite group and let $Y$ be a Stein $G$-manifold.  If $Y$ is $G$-Oka, then $Y$ is $G$-elliptic.
\end{corollary}

\begin{proof}
The proof is an adaptation of the standard proof in \cite[p.~230]{Forstneric2017}, using Theorem \ref{t:main} with jet interpolation.  The tangent bundle $X$ of $Y$ is Stein and the projection $X\to Y$ is equivariant with respect to the induced $G$-action on $X$.  Identify $Y$ with the zero section of $X$.  As in the proof of \cite[Proposition 3.3.1]{Forstneric2017}, using equivariant Stein theory (see Lemma \ref{l:equivariant-Stein-theory} below and its proof), we can produce a $G$-stable neighbourhood $V$ of $Y$ in $X$ and a holomorphic $G$-map $s:V\to Y$ such that $s(y)=y$ and $d_y s$ maps $T_y Y\subset T_y X$ surjectively onto $T_y Y$ for all $y\in Y$.  

We may assume that $s$ extends to a continuous $G$-map $X\to Y$.  Namely, take a $G$-invariant Riemannian metric on $Y$ with norm $\lVert\cdot\rVert$ such that the closure of the neighbourhood $U$ of $Y$ in $X$ defined by $\lVert x\rVert\leq 1$ lies in $V$.  Equivariantly retract $X$ onto $\overline U$ by the formula $x\mapsto x/\lVert x\rVert$ for $x\in X\setminus U$ and precompose $s\vert \overline U$ by the retraction.

By Theorem \ref{t:main}, $s$ may be deformed to a holomorphic $G$-map $X\to Y$ that agrees with $s$ to first order along $Y$, giving a dominating $G$-spray on $Y$. 
\end{proof}

\section{Equivariant Oka principle for reductive group actions}   \label{s:reductive-action}

\noindent
In this section, we present generalisations of Theorem \ref{t:main} to actions of reductive complex Lie groups $G$ that are not necessarily finite.  We first consider the case that all the $G$-orbits in the source Stein $G$-manifold $X$ are closed.  This is equivalent to all slice representations $L\to\GL(W)$ having finite image in $\GL(W)$.  If $X/G$ is connected (equivalently, irreducible), then there is a principal isotropy group,\footnote{We use the terms \textit{isotropy group} and \textit{stabiliser} interchangeably.} and if it is finite, then all isotropy groups are finite. 

We first prove a version of the Luna-Richardson theorem \cite{LunaRichardson} for arbitrary Stein $G$-manifolds.  Assuming that $X\git G$ is connected, let $H$ be the principal isotropy group of $X$ and let $X^H_0$ denote the union of the connected components of $X^H$ which intersect $X_\pr$, the set of principal orbits (closed orbits with isotropy group conjugate to $H$).  Then $N=N_G(H)/H$ acts on $(X_\pr)^H$ and preserves its closure $X^H_0$.

\begin{lemma}   \label{lem:LR}
Suppose that $Y=G\times^LW$, where $L$ is a reductive closed subgroup of $G$ and $H$ is a principal isotropy group of the $L$-module $W$ (hence a principal isotropy group of $Y$).    Then the canonical map $\rho: Y^H_0\git N\to Y\git G$ is an algebraic isomorphism.
\end{lemma}

\begin{proof}
We apply \cite[Theorem 2.2]{LunaRichardson}. For this we need to establish the following three statements.
\begin{enumerate}
\item For each principal fibre $F=G\times^HU$, the intersection $F\cap Y^H_0$ contains a unique closed $N$-orbit.  Note that $U$ is an $H$-module with $U\git H=U^H=\{0\}$.
\item $Y^H_0\git N$ is irreducible.
\item Each closed $G$-orbit in $Y$ intersects $Y^H_0$.
\end{enumerate}
For (1) we may assume that $H=G_y$, where $y\in  Y_\pr\cap Y^H\subset Y^H_0$. Now $(G/H)^H = N_G(H)/H$, so $(G\times^HU)^H= N_G(H)\times^H\{0\}=Ny$, which is closed, giving (1).   Moreover, we have shown that the inverse image of $(Y\git G)_\pr$ under $\rho$ is $Y^H\cap Y_\pr$, which is open and dense in $Y^H_0$. Thus $Y^H_0\git N$ has a dense open subset isomorphic to $(Y\git G)_\pr$, which is irreducible, and we have (2).  Finally, let $Gy$ be a closed orbit in $Y$. Then by changing $y$ by an element of $G$ we can arrange that $H$ is the principal isotropy group of the slice representation of $G_y$ at $y$. Thus $y\in Y^H_0$, giving (3).
\end{proof}

\begin{proposition}[Luna-Richardson \cite{LunaRichardson}]\label{prop:LR}
Let $X$ be a Stein $G$-manifold such that $X\git G$ is connected. Then restriction to $X^H_0$ induces an isomorphism $X^H_0\git N\to X\git G$.
\end{proposition}

\begin{proof}
This follows from Lemma \ref{lem:LR} since $X$ is locally $G$-biholomorphic to tubes $G\times^LS$, where $S$ is an $L$-saturated neighbourhood of $0$ in an $L$-module $W$ with principal isotropy group $H$.
\end{proof}

\begin{proposition}\label{prop:induced}
 Let $X$ be a Stein $G$-manifold with all $G$-orbits closed and $X/G$ connected.  Then the natural map $G\times^{N_G(H)} X^H_0\to X$, $[g,x]\to gx$, is a $G$-biholomorphism. 
\end{proposition}

\begin{proof}  
Since all orbits are closed, the result follows from Proposition \ref{prop:LR}. 
\end{proof}

Note that the principal isotropy group of $X_0^H$ relative to $N_G(H)/H$ is trivial and all orbits are closed.

\noindent
{\bf Erratum added 29 Sep 2023.}  Lemma 5.4 is false. Here is a counterexample due to P.~Heinzner. Let $X=\C^*$, $Y=\C$, both with the multiplicative action of $G=\C^*$. Let $S=X\times\{0\}$. Then $S$ is Stein and $G$-stable, but the only $G$-stable neighbourhood of $S$ is $X\times Y$.

Lemma 5.4 is used in Proposition 5.5 and the proof of Theorem 5.6. 

In Proposition 5.5, $\Omega$ is a $G$-stable neighbourhood of $X$ in $U$ which we want to shrink to be Stein. Since $\Omega$ consists of closed orbits of dimension $\dim G$, it is $G$-saturated. Hence $\Omega/G$   is a neighbourhood of $X/G$ in $V\git G$. We can just replace  $\Omega$ by the inverse image of a Stein subneighbourhood of $X/G$ in $\Omega/G$.

In Theorem 5.6, we only use Lemma 5.4 to show that the graph $\Gamma_f$ of a $G$-equivariant holomorphic map $f\colon X\to Y$ has \emph{some\/} Stein $G$-stable neighbourhood.  Since all $G$-orbits in $X$ are closed with finite stabiliser, the quotient  of $X\times Y$ by $G$ exists. Let $U$ be a Stein neighbourhood of $\Gamma_f/G$ in $(X\times Y)/G$. Then the inverse image of $U$ is the required $G$-stable Stein neighbourhood of $\Gamma_f$. 

\begin{lemma}   \label{lem:smallGStein} 
Let $X$ be a Stein $G$-manifold all of whose isotropy groups are finite.  Let $Y$ be a complex $G$-manifold and let $S$ be a $G$-stable Stein subvariety of $X\times Y$.  Let $U$ be a  neighbourhood of $S$ in $X\times Y$.  Then there is a $G$-stable Stein neighbourhood $U'$ of $S$ contained in $U$.
\end{lemma}

\begin{proof}
By the holomorphic slice theorem, $(X\times Y)/G$ is locally a quotient by finite groups, hence is a complex space, and $S/G$ is Stein in $(X\times Y)/G$.  Now $U^c$ is closed and, since the action of $G$ is proper, $GU^c$ is closed and $G$-stable, hence its image $B$ in $(X\times Y)/G$ is closed and does not intersect $S/G$.  Thus there is a Stein neighbourhood of  $S/G$ in $B^c$ and its inverse image in $X\times Y$ is the required $U'$.
\end{proof}

Now we need an equivariant tubular neighbourhood result.  Let $K$ be a maximal compact subgroup of $G$.

\begin{proposition}   \label{prop:tubularneighbourhood} 
Assume that $X$ is a closed $G$-stable submanifold of a $G$-module $V$ and that all the isotropy groups of $X$ are finite.  Then there is a $G$-stable neighbourhood $\Omega$ of $X$ in $V$ and a deformation retraction $h_t : \Omega\to \Omega$ of $\Omega$ to $X$, $t\in[0,1]$, such that each $h_t$ is holomorphic and $G$-equivariant.
\end{proposition}

\begin{proof}
We have the trivial bundle $X\times V$ and the subbundle $TX$. There is a $K$-stable complement $\nu$ to $TX$ in $X\times V$, which we may assume is a holomorphic $G$-vector bundle. Since the $G$-action on $X$ is proper, there is a $G$-invariant norm $\lVert\cdot\rVert$ on $\nu$. 

Let $\Theta: \nu\to V$, $(y,v)\mapsto y+v$.  Then $\Theta$ is $G$-equivariant and is a biholomorphism from a neighbourhood $\Omega'$ of the zero section of $\nu$ onto a neighbourhood $\Omega$ of $X$.  Let  $U\subset V$ be the set of points with closed orbit and finite stabiliser, an open $G$-invariant neighbourhood of $X$.  We may assume that $\Omega\subset U$.  By Lemma \ref{lem:smallGStein}, we may assume that $\Omega$ is $G$-invariant (and Stein), hence that $\Omega'$ is $G$-invariant and Stein.  Then $\Theta$ is a $G$-biholomorphism of $\Omega'$ onto $\Omega$.  Choose a $G$-invariant continuous function $f: X\to (0,\infty)$ such that 
\[  \{v\in\nu_x : \lVert v\rVert<f(x)\}\subset \Omega'. \]
By shrinking $\Omega'$ (hence also $\Omega$), we can assume that there is equality. Then there is a deformation retraction of $\Omega'$ to the zero section by holomorphic $G$-equivariant maps (given by scalar multiplication by elements of $[0,1]$), hence there is a deformation retraction of $\Omega$ onto $X$ by holomorphic $G$-maps.
\end{proof}

\begin{theorem}   \label{t:second-main} 
Let $G$ be a reductive complex Lie group and let $K$ be a maximal compact subgroup of $G$.  Let $X$ be a Stein $G$-manifold all of whose isotropy groups are finite.  Let $Y$ be a $G$-Oka manifold.  Then every $K$-equivariant continuous map $f:X\to Y$ is homotopic, through such maps, to a $G$-equivariant holomorphic map. 

{\rm (a)}  If $f$ is holomorphic on a $G$-invariant subvariety $Z$ of $X$, then the homotopy can be chosen to be constant on $Z$.

{\rm (b)}  If $f$ is holomorphic on a neighbourhood of a $G$-invariant subvariety $Z$ of $X$ and on a neighbourhood of a $K$-invariant $\O(X)$-convex compact subset $A$ of $X$, and $\ell\geq 0$ is an integer, then the homotopy can be chosen so that the intermediate maps agree with $f$ to order $\ell$ along $Z$ and are uniformly close to $f$ on $A$.
\end{theorem} 

\begin{remark}
Note that we assume that all isotropy groups of $X$ are finite.  Suppose that we only knew that all the $G$-orbits in $X$ are closed, but we started with a $G$-equivariant continuous map $f: X\to Y$. Then using Proposition \ref{prop:induced} we would be able to reduce to the case that all isotropy groups of $X$ are finite and Theorem \ref{t:second-main} would apply.
\end{remark}

\begin{proof} [Proof of Theorem \ref{t:second-main}]
We may assume that $X$ is connected.  To begin with we assume that $X/G$ has finitely many Luna strata.

We follow the proof of Theorem \ref{t:main} to establish (a).  Define the $Q_k$ as before and consider Claim 1.  The graph $\Gamma$ of $f_{k-1}$ in $X\times Y$ is Stein.  By Lemma \ref{lem:smallGStein}, $\Gamma $ has a $G$-invariant Stein neighbourhood $V$ in $X\times Y$.  Since $X$ has only finitely many conjugacy classes of isotropy groups (which are finite), we may assume the same for $V$.  Since $V$ is smooth, this implies that there are finitely many Luna strata, hence $V$ has a closed $G$-equivariant embedding into a $G$-module $\C^N$.  By Proposition \ref{prop:tubularneighbourhood}, $V\subset \C^N$ has a $G$-invariant tubular neighbourhood $W$ with a $G$-equivariant holomorphic deformation retraction $W\to V$.  Let $U$ be a Stein $G$-neighbourhood of $\pi\inv(Q_{k-1})\cup Z$ in $X$.  Consider the holomorphic map from $\pi^{-1}(Q_{k-1}) \cup Z$ to $\C^N$ given by $x\mapsto (x, f_{k-1}(x))$.  By the Cartan extension theorem, it can be extended to a holomorphic map $U\to\C^N$ and made $G$-equivariant by averaging over $K$.  After further shrinking $U$ if necessary, the extension will take $U$ into $W$.  Then, as before, we obtain the desired extension $f_k': U\to Y$ of $f_{k-1}$.  Let $I$ denote $[0,1]$ and let $U_0$ be a $G$-stable neighbourhood of $\pi\inv(Q_{k-1})\cup Z$ with closure in $U$.  Consider the closed $G$-stable subset $A=A_1\cup A_2\cup A_3$ of $X\times I$ where
\[  A_1=X\times\{0\},\quad A_2= (\pi\inv(Q_{k-1})\cup Z)\times I,\quad A_3= \overline{U_0}\times\{1\}. \]
We have a continuous $K$-equivariant map from $A$ to $Y$ which is $f$ on $A_1$, the homotopy of $f$ and $f_{k-1}$ on $A_2$, and $f'_k$ on $A_3$.  By \cite[Theorems 6.4 and 8.8]{Murayama1983}, the map extends to a continuous $K$-equivariant map on a $K$-stable neighbourhood of $A$.  This neighbourhood contains one of the form $U'\times I$ where $U'$ is a $G$-stable neighbourhood of $\pi\inv(Q_{k-1})\cup Z$.  This gives us a homotopy of $f$ and $f'_k$ on $U'$ which extends the homotopy of $f$ and $f_{k-1}$.  This concludes the proof of Claim 1.

Now we consider Claim 2.  We will make use of the Kempf-Ness set $R$ of $X$ (see \cite{HH1994}, \cite[\S 12, Retraction Theorem]{HK1995}).  It enjoys the following properties (for any reductive group $G$ and any Stein $G$-manifold $X$).
\begin{itemize}
\item $R$ is a closed $K$-stable subset of $X$.
\item There is a $K$-equivariant deformation retraction of $X$ onto $R$.
\item If $x\in R$, then $G_x=(K_x)_\C$.  In particular, if $G_x$ is finite, then $G_x=K_x$.
\item Every closed $G$-orbit in $X$ intersects $R$ in a unique $K$-orbit and $R/K$ and $X\git G$ are naturally homeomorphic.
\item If all $G$-isotropy groups of $X$ are finite, then $X$ is $G$-homeomorphic to $G\times^K R$.
\end{itemize}
Define $E$ as before and reduce to the case that $f$ is holomorphic on a $G$-invariant neighbourhood $U''$ of $\pi\inv(Q_{k-1})\cup Z$ in $\pi\inv(Q_k)\cup Z$.  We obtain continuous $G$-maps $h_t: \pi\inv(S_k)\cup Z\to Y$, holomorphic on $U''$. Restrict the $h_t$ to $(\pi\inv(Q_k)\cup Z)\cap T_B$, where $T_B$ is a tube $G\times^LB$, $L=G_x=K_x$, $x\in R$, and $B$ is a ball in an $L$-module.  Since $L$ is finite, we are in the situation of Theorem \ref{t:main}, and Riemann's extension theorem shows that the $h_t$ extend holomorphically across $\pi\inv(Q_{k-1})\cap T_B$ along with the homotopy, completing the proof of Claim 2 and (a).

The reduction of (b) to (a) follows that in Theorem \ref{t:main}.  There is a $G$-invariant Stein neighbourhood $U$ of $Z\cup A$ on which $f$ is holomorphic.  Let $\phi: U\to\C^N$ be an equivariant closed embedding into a $G$-module. Then $M$, $j$, $\pi$, and the $h_t$ are as before.  We can approximate $\phi$ uniformly on $A$ by a holomorphic $G$-map $\psi: X\to \C^N$ which agrees with $\phi$ to order $\ell$ on $Z$ and the proof concludes as before. 

We now drop the assumption that $X/G$ has finitely many Luna strata.  Consider (a).  Let $B\subset B'\subset U\subset R$ where $B$, $B'$ are compact and $K$-stable and $U$ is $K$-stable and relatively compact. Since $GU/G$ has finitely many Luna strata, there is a homotopy $f_t$ of $f$ on $GU$ which is constantly $f$ on $Z\cap GU$,   uniformly close to $f$ on $B$ with $f_1$ holomorphic on $GU$. Using a continuous $[0,1]$-valued $G$-invariant function which is $1$ on a neighbourhood of $GB'$ and $0$ in a neighbourhood of $GU^c$, we obtain a new homotopy $f_t$ of $f$ on $X$ which equals $f$ on $Z$,  is uniformly close to $f$ on $B$, and is holomorphic on a neighbourhood of $GB'$.

Let $B_0=\emptyset\subset B_1\subset B_2\subset \cdots$ be a sequence of  $K$-stable compact subsets of $R$ whose union is $R$.  By the above, there is a homotopy $f_{(1,t)}$ of $f$ so that $f_{(1,1)}$ is holomorphic on a neighbourhood of $GB_1$ and equal to $f$ on $Z$.  By induction  there are homotopies $f_{(j,t)}$ of $f_{(j-1,1)}$ such that:
\begin{itemize}
\item $f_{(j,t)}$ is holomorphic on a neighbourhood of $GB_j$ for all $t\in[0,1]$.
\item  $d(f_{(j-1,1)}(x)-f_{(j,t)}(x))<2^{-j-1}$ for  $x\in B_{j-1}$ and $t\in[0,1]$, where $d$ is a complete metric on $Y$ inducing its topology.
\end{itemize}
For $s\in[1,\infty)$, choose an integer $j\geq 1$ such that $0\leq s-j<1$ and let $f_s(x)=f_{(j,s-j)}(x)$ for $x\in X$. 
Then $f_\infty=\lim\limits_{s\to\infty} f_s$ is the uniform limit on $B_j$ of $G$-maps which are holomorphic on a neighbourhood of $GB_j$.  Since $X=G\times^KR$, the convergence is uniform on compact subsets of $X$.  Thus $f_\infty$ is holomorphic.  Now changing the parameter space $[0,\infty]$ to $[0,1]$ gives the required homotopy.  This establishes (a) and a similar argument gives (b).
\end{proof}

In future work, we hope to be able to generalise Theorem \ref{t:second-main} to the case of arbitrary reductive group actions.  This will require new methods.  For now, we offer the following result in this direction.

\begin{theorem}   \label{t:one-slice-type}
Let $G$ be a reductive complex Lie group and $K$ a maximal compact subgroup of $G$.  Let $X$ be a Stein $G$-manifold with a single slice type and $Y$  a $G$-Oka manifold.  Then every $K$-equivariant continuous map $X\to Y$ is homotopic, through such maps, to a $G$-equivariant holomorphic map.
\end{theorem}

The theorem is an immediate consequence of the next proposition.

\begin{proposition}   \label{p:stratum-stein}
Let $G$ be a reductive group and $K$ a maximal compact subgroup of $G$.  Let $X$ be a Stein $G$-manifold and $Y$ a $G$-Oka manifold.  Let $S$ be a Luna stratum of $X\git G$ which is Stein.  Then every $K$-equivariant continuous map $\pi^{-1}(S)\to Y$ is homotopic, through such maps, to a $G$-equivariant holomorphic map.
\end{proposition}

The proof of the proposition is by a series of lemmas.  Let $(W,H)$ be the slice representation associated to $S$.  Here, $H$ is a reductive closed subgroup of $G$.  Let $W=W^H\oplus W'$ be an $H$-module decomposition. Then the fibres of $\pi\inv(S)\to S$ are isomorphic to $F:=G\times^H\N(W')$. Thus the $G$-fibre bundle $\pi\inv(S)\to S$ has structure group $L:=\Aut(G\times^H\N(W'))^G$ (which is a linear algebraic group) and associated principal $L$-bundle $\L$.   We know from \cite[Corollary 6.31]{KLS2017} that $\Aut_\vb(G\times^HW')^G$ ($\vb$ stands for $G$-vector bundle automorphisms) is reductive. Let $L_\vb$ denote this group. (Note that this is not the same $L_\vb$ as in \cite{KLS2017}.)

\begin{lemma}
We have $L_\vb\subset L$ and $L_\vb$ is a Levi component of $L$, that is, $L= L_\vb\ltimes P$, where $P$ is the unipotent radical of $L$.
\end{lemma}

\begin{proof}
There is a homomorphism $\delta: L\to L_\vb$, where $\delta(\ell)$ for $\ell\in L$ is the normal derivative of $\ell$ along the closed orbit $G\times^H\{0\}$.  Here we use the fact that the Zariski tangent space of $\N(W')$ at $0$ is $W'$.  Now $ L_\vb$ acts by automorphisms of the invariants $\O_\textrm{alg}(G\times^HW')^G$, preserving those vanishing on $G/H$, whose zero set is $F$.  Hence $L_\vb\subset L$ and  
\[ L_\vb\to L\xrightarrow{\delta}L_\vb \] 
is the identity.

Let $M$ be a Levi factor of $L$ so that $L=M\ltimes P$.  Let $C$ denote the kernel of  $\delta: L\to L_\vb$.  Since $C$ is normal in $L$, $M_0:=C\cap M$ is normal in $M$, hence reductive.  Since $\delta(M_0)$ is the identity, $M_0$ fixes $G/H$ and acts trivially on $W'$, the Zariski tangent space of $\N(W')$ at $0\in W'$.  By Luna's slice theorem, $M_0$ acts trivially on $G\times^HW'$, so $M_0=\{e\}$.  Thus $\delta: M\to L_\vb$ is an isomorphism and $L= L_\vb\ltimes P$.
\end{proof}

Let $X_S$ denote the set of closed $G$-orbits in $\pi\inv(S)$.  Then $X_S\to S$ is a $G$-fibre bundle over $S$ with fibre $G/H$ and structure group $\Aut(G/H)^G\simeq N=N_G(H)/H$ acting on the right on the fibre.

\begin{corollary}
We can reduce the structure group of $\L$ to $L_\vb$ and there is a deformation retraction of $\pi\inv(S)\to X_S$ by $G$-equivariant holomorphic maps.
\end{corollary}

\begin{proof}
The fibre of $\L/L_\vb$ is isomorphic to $P$, which is isomorphic to its Lie algebra, hence contractible.  There is a continuous section $\sigma$ of $\L/L_\vb\to S$. Since $S$ is Stein, we may assume that $\sigma$ is holomorphic.  Then for $x\in S$, $\sigma(x)$ is the choice of an $L_\vb$-orbit  in the fibre of $\L$ over $x$. This is a holomorphic principal $L_\vb$-bundle $E$  and clearly $\L\simeq E\times^{L_\vb}L$, so that $E$ is a holomorphic reduction of the structure group of $\L$ to $L_\vb$. Then we have a holomorphic $G$-equivariant isomorphism of $\pi\inv(S)\to S$ to the associated bundle $\F:=E\times^{L_\vb}F\to S$. On the  latter bundle we have a $G$-equivariant deformation retraction by holomorphic maps given by multiplication by $t\in[0,1]$ on $G\times^H\N(W')\subset G\times^HW'$. 
\end{proof}

Using our $G$-equivariant holomorphic deformation retraction we have a deformation retraction of $\mathscr C^K(\pi\inv(S),Y)$ to $\mathscr C^K(X_S,Y)$ and of $\O^K(\pi\inv(S),Y)$ to $\O^K(X_S,Y)$.  Recall that $X_S$ is just what we get in $\pi\inv(S)$ by replacing the fibre $F$ by its closed orbit $G\times^H\{0\}\simeq G/H$.  Let $X_S^H$ be shorthand for $(X_S)^H$.  Note that $N$ acts on the left on $X_S^H$ and $X_S^H\to S$ is a principal $N$-bundle. Moreover,  $X_S$ is canonically isomorphic to $(G/H)\times^N X_S^H$. 

\begin{lemma}
There is a continuous $K$-equivariant deformation retraction of $\mathscr C^K(X_S,Y)$ to a closed subset of $\mathscr C^G(X_S,Y)$.
\end{lemma}

\begin{proof}
Let $L\subset K$  be a maximal compact subgroup of $H\subset G$, respectively.  The maximal compact subgroup of $N=N_G(H)/H$ is $N_0=N_K(L)/L$.  Since $N/N_0$ is contractible, there is a reduction of the structure group of $X_S^H\to S$ to $N_0$.  Thus there is a principal $N_0$-bundle $P_0$ (with $N_0$ acting on the left) such that $X_S^H\simeq N \times^{N_0} P_0$. 

Let $\mathfrak k$ be the Lie algebra of $K$ and $\mathfrak l$ that of $L$. Then $G\simeq K\times \exp(i\mathfrak k)$ and $H\simeq L\times \exp(i\mathfrak l)$. Since $\mathfrak l\subset\mathfrak k$,   the obvious $K$-deformation  retraction of $G$ to $K$ sends $H$ to $L$. Thus we have a $K$-equivariant deformation retraction of $G/H$ to $K/L$, where $N_0$ acts on $K/L$.  Hence we have a $K$-equivariant deformation retraction of the $G/H$-bundle $X_S\to S$ to the $K/L$-bundle $X_{S,0}:=K/L\times^{N_0}P_0\to S$. Then starting with $f\in \mathscr C^K(X_S,Y)$, we have a $K$-equivariant deformation retraction of $f$ to an element $f'\in \mathscr C^K(X_{S,0},Y)$. This is the same thing as an $N_0$-equivariant map of $P_0$ to $Y^L=Y^H$.  But given such an $N_0$-equivariant map, it naturally gives rise to an $N$-equivariant map of $N\times^{N_0}P_0=X_S^H\to Y^H$, which is the same thing as an element $f_0\in \mathscr C^G(X_S,Y)$. It is easy to see that our deformation retraction of $\mathscr C^K(X_S,Y)$ to $\mathscr C^K(X_{S,0},Y)$ sends $f_0$ back to $f'$. Thus we have the desired $K$-equivariant deformation retraction.
\end{proof}

Let $E$ denote the $Y^H$-bundle $(X_S^H  \times Y^H)/N$ over $X_S^H/N=S$.  Let $\Gamma_{\!\mathscr C}(E)$ (resp.\ $\Gamma_{\!\O}(E)$) denote the space of continuous (resp.\ holomorphic) sections of $E$.  We give $\Gamma_{\!\mathscr C}(E)$ and its closed subspace $\Gamma_{\!\O}(E)$ the usual topology, and similarly for $\mathscr C^G(X_S,Y)$ and $\O^G(X_S,Y)$.

\begin{lemma}
There is a homeomorphism between $\mathscr C^G(X_S,Y)$ and $\Gamma_{\!\mathscr C}(E)$ which sends $\O^G(X_S,Y)$ onto $\Gamma_{\!\O}(E)$.
\end{lemma}

\begin{proof}
Let $f\in \mathscr C^G(X_S,Y)$. Then $f$ is determined by its restriction $f'$ to $X_S^H$, where $f'\in C^N(X_S^H,Y^H)$ (and vice versa). Consider the map $X_S^H\to\Gamma(f')$, where $\Gamma(f')\subset X_S^H\times Y^H$ is the graph of $f'$. Since $\Gamma(f')$ is $N$-stable, quotienting by $N$ we obtain  a continuous section of $E$. Conversely, let $\sigma\in\Gamma_{\!\mathscr C}(E)$. Since $X_S^H\to S$ and $E$ are locally trivial, locally over $S$, $\sigma$ corresponds to an element of $\mathscr C^N(X_S^H,Y^H)$. But these local equivariant maps are unique and patch together to form an element of $\mathscr C^N(X_S^H,Y^H)$ which induces $\sigma$. 
\end{proof}

\begin{proof}[Proof of Proposition \ref{p:stratum-stein}]
It is enough to show that every element of $\mathscr C^G(X_S,Y)$ is homotopic to an element in $\O^G(X^S,Y)$.  But $\mathscr C^G(X_S,Y)\simeq \Gamma_{\!\mathscr C}(E)$ and since $Y^H$ is Oka, any element of $\Gamma_{\!\mathscr C}(E)$ is homotopic to a holomorphic section, which is the same thing as an element of $\O^G(X_S,Y)$.  
\end{proof}

\section{Equivariant Runge property}  \label{s:equivar-Runge}

\noindent
Let $G$ be a reductive complex Lie group.  We say that a $G$-manifold $Y$ is $G$-\emph{Runge} if whenever $X$ is a Stein $G$-manifold and $\Omega$ is a $G$-saturated Runge domain in $X$, the closure of the image of the restriction map $\rho_\Omega^X:\O^G(X,Y)\to\O^G(\Omega,Y)$ is a union of path components (possibly empty).  Roughly speaking, this means that approximability of holomorphic $G$-maps $\Omega\to Y$ by holomorphic $G$-maps $X\to Y$ is deformation invariant.

When $G$ is the trivial group, the $G$-Runge property of $Y$ is one of the equivalent formulations of the Oka property.  It is a variant of the so-called homotopy Runge property, it implies the Oka property formulated as the convex approximation property, and is implied by the Oka property formulated as the basic Oka property with approximation.

We recall that to say that $\Omega$ is Runge means that $\Omega$ is Stein and the image of the restriction map $\O(X)\to\O(\Omega)$ is dense.  Equivalently, $\Omega$ can be exhausted by compact subsets that are $\O(X)$-convex (and not merely $\O(\Omega)$-convex).

\begin{proposition}  \label{p:Runge-G-domains}
Let $G$ be a reductive complex Lie group and let $K$ be a maximal compact subgroup of $G$.  For a $G$-saturated Stein domain $\Omega$ in a Stein $G$-manifold $X$, the following are equivalent.
\begin{enumerate}
\item[(i)]  $\Omega$ is Runge.
\item[(ii)]  The image of the restriction map $\O^G(X)\to\O^G(\Omega)$ is dense.
\item[(iii)]  $\Omega$ can be exhausted by $K$-invariant $\O(X)$-convex compact subsets of $X$.
\item[(iv)]  $\Omega$ is the preimage of a Runge domain in $X\git G$.
\end{enumerate}
\end{proposition}

The proposition shows that the domains $\Omega$ in the definition of the $G$-Runge property are precisely the preimages in $X$ of the Runge domains in the normal Stein space $X\git G$.  The proposition also shows that there is no new \lq\lq$G$-Runge property\rq\rq\ for $G$-saturated Stein domains in $X$.  There is just the ordinary Runge property.  Thus there is no property on the source side that should be called \lq\lq $G$-Runge\rq\rq, instead of using that term on the target side as we have done.

\begin{proof}
(i) $\Rightarrow$ (ii) by the Oka-Weil approximation theorem and an averaging argument using $K$.

(ii) $\Rightarrow$ (iv):  By (ii), the equivalence relations defined on $\Omega$ by $\O^G(\Omega)$ and $\O^G(X)$ are the same, so $\pi(\Omega)=\Omega\git G$, and $\pi(\Omega)$ is Stein.  Also, $\O(X\git G)=\O^G(X)$ is dense in $\O(\pi(\Omega))=\O^G(\Omega)$, so $\pi(\Omega)$ is Runge in $X\git G$.

(iv) $\Rightarrow$ (iii):  Exhaust the Runge domain $\pi(\Omega)$ by $\O(X\git G)$-convex subsets.  Take their preimages in $X$ and intersect with $K$-invariant $\O(X)$-convex compact subsets that exhaust $X$.

(iii) $\Rightarrow$ (i) is obvious.
\end{proof}

Next we show that for Runge domains, $G$-stability and $G$-saturation are in fact the same property.

\begin{proposition}
Let $G$ be a reductive complex Lie group, $X$ a Stein $G$-manifold, and $\Omega$ a $G$-stable Stein open subset of $X$ which is Runge in $X$.  Then $\Omega$ is $G$-saturated.
\end{proposition}

\begin{proof}
Let $F$ be a fibre of the quotient map $X\to X\git G$ which intersects $\Omega$.  We show that $F\subset\Omega$.  Let $Gx$ be the closed orbit in $F$.  Then $F\simeq G\times^H\N(W)$, where $H=G_x$ and $(W,H)$ is the slice representation at $x$.  There is an $H$-saturated neighbourhood $S$ of the origin in $W$ such that there is a $G$-saturated Stein neighbourhood of $Gx$, $G$-biholomorphic to $G\times^HS$.  Using this biholomorphism we can reduce to the case that $X=\N(W)$ and $\Omega$ is an $H$-stable Stein open subset of $X$.  Since $\O(X)^H=\C$ is dense in $\O(\Omega)^H$, $\O(\Omega)^H=\C$ also.

Now the spaces $\O_\gf(X)$ and $\O_\gf(\Omega)$ of $H$-finite functions are dense in $\O(X)$ and $\O(\Omega)$, respectively.  The $H$-finite functions on $X$ are just the polynomial functions on $X$.  Since $\Omega$ is Runge in $X$, the $H$-finite functions on $X$ are dense in the $H$-finite functions on $\Omega$.  But the multiplicity of any irreducible $H$-module in $\O_\gf(X)$ and $\O_\gf(\Omega)$ is finite.  Hence $\O_\gf(X)$ maps onto $\O_\gf(\Omega)$.  Therefore $X$ and $\Omega$ are $H$-equivariantly biholomorphic.  Since $X$ has a unique $H$-fixed point $0\in S$, so does $\Omega$, and we see that $0\in\Omega$.  Since every $H$-orbit in $X$ has $0$ in its closure, $H\Omega=\Omega=X$, establishing the proposition.
\end{proof}

Here are the basic properties of the $G$-Runge property.  Note the similarity to Propositions \ref{p:Oka-properties} and \ref{p:elliptic-properties}.

\begin{proposition}  \label{p:G-Runge-properties}
\begin{enumerate}
\item  If $G$ acts trivially on $Y$, then $Y$ is $G$-Runge if and only if $Y$ is Oka.
\item  If $Y$ is $G$-Runge, then $Y$ is $H$-Runge for every reductive closed subgroup $H$ of $G$.
\item  If $Y$ is $G$-Runge, then $Y$ is Oka.
\item  If $Y_k$ is $G_k$-Runge, $k=1,2$, then $Y_1\times Y_2$ is $G_1\times G_2$-Runge.  
\item  If $Y_1$ and $Y_2$ are $G$-Runge, then $Y_1\times Y_2$ is $G$-Runge with respect to the diagonal action.
\item  A holomorphic $G$-retract of a $G$-Runge manifold is $G$-Runge.
\item  If $Y$ is $G$-Runge, then $Y$ is $G$-Oka.
\item  Conversely, if $Y^H$ is an Oka submanifold of $Y$, then $Y$ has the $G$-Runge property with respect to Stein $G$-manifold sources with single orbit type and stabiliser $H$.
\end{enumerate}
\end{proposition}

\begin{proof}
(1)  If $G$ acts trivially on $Y$ and $X$ is a Stein $G$-manifold, then a holomorphic $G$-map $f:X\to Y$ is $G$-equivariant if and only if it is $G$-invariant, so $f$ factors uniquely through $X\git G$.  Thus the $G$-Runge property for holomorphic maps into $Y$ from $X$ and $\Omega$ is equivalent to the homotopy Runge property for holomorphic maps into $Y$ from $X\git G$ and $\Omega\git G$.  As already remarked, the homotopy Runge property is one of the equivalent formulations of the Oka property.

(2) If $X$ is a Stein $H$-manifold, then $\O^G(G\times^H X, Y)$ is naturally homeomorphic to $\O^H(X,\res_H^G Y)$.  The $G$-manifold $G\times^H X$ is Stein, and there is a natural bijection between $H$-saturated Runge domains in $X$ and $G$-saturated Runge domains in $G\times^H X$, with $\Omega\subset X$ corresponding to $G\times^H \Omega\subset G\times^H X$.

(3) follows from (1) and (2).

(4--6) are easy.

(7) holds by (2) and because a $G$-map from a Stein manifold $X$ with a trivial $G$-action into $Y$ is the same as a plain map from $X$ into $Y^G$.

(8)  First consider the case when $H$ is trivial, so the $G$-action on $X$ is free and $X$ is a principal $G$-bundle over a Stein geometric quotient $X/G$.  Then there is a natural bijection between $G$-maps $X\to Y$ and sections of a bundle over $X/G$ with fibre $Y$.  Also, as shown above, the $G$-saturated Runge domains in $X$ are precisely the preimages of the Runge domains in $X/G$.  The conclusion now follows from the plain Oka property for sections of this bundle.  

In general, by Proposition \ref{prop:induced}, the natural map $G\times^{N} X^H\to X$, where $N=N_G(H)$, is a $G$-biholomorphism, and $L=N/H$ acts freely on $X^H$.  Thus $X^H$ is a principal $L$-bundle.  By adjunction,
\[ \O^G(X,Y) \simeq \O^N(X^H, \res_N^G Y) \simeq \O^L(X^H, Y^H), \]
so we are reduced to the case of a free action on the source.
\end{proof}

Next we present two classes of examples.

\begin{proposition}  \label{p:G-Runge-examples}
\begin{enumerate}
\item  A $G$-module is $G$-Runge.
\item  $G$ itself is $G$-Runge.
\end{enumerate}
\end{proposition}

We do not know a simple proof that a $G$-homogeneous space is $G$-Runge.  In the next section, we show that $G$-ellipticity implies $G$-Runge; it follows that $G$-homogeneous spaces are $G$-Runge.

\begin{proof}
(1)  If $Y$ is a $G$-module and $\Omega\subset X$ are as in the definition of the $G$-Runge property, then by the Oka-Weil approximation theorem and an averaging argument using $K$, the closure of $\rho_\Omega^X(\O^G(X,Y))$ is all of $\O^G(\Omega,Y)$.

(2)  Let $X$ be a Stein $G$-manifold.  Let $f:X\to G$ be a holomorphic $G$-map.  The existence of $f$ implies that the $G$-action on $X$ is free, so $X$ is a principal $G$-bundle over the Stein base $X/G$.  The existence of $f$ also implies that $X\to X/G$ is holomorphically trivial, so a $G$-map from $X$ is nothing but a plain map from $X/G$.  The Oka property of $G$ now implies that $G$ is $G$-Runge.
\end{proof}

\section{$G$-elliptic implies $G$-Runge}  \label{s:ellipticity-implies-Runge}

\noindent
The proof of the following theorem shows the essence of Gromov's linearisation method as sketched in \cite[Section 1.4]{Gromov1989} (see also \cite[Remark 6.6.4]{Forstneric2017}), here adapted to the presence of a group action.

\begin{theorem}  \label{t:G-ell-implies-G-Runge}
Let $G$ be a reductive complex Lie group.  A $G$-elliptic manifold is $G$-Runge.
\end{theorem}

We are unable to replace the $G$-ellipticity assumption by the $G$-Oka property.  If we could, it would show that the $G$-Oka and $G$-Runge properties are equivalent, which we expect to be the case.  To prove the theorem we need a lemma with some basic equivariant Stein theory.  For lack of a reference we sketch a proof.

\begin{lemma}  \label{l:equivariant-Stein-theory}
Let $K$ be a compact real Lie group.  Let $X$ be a Stein $K$-manifold.  In the category of holomorphic $K$-vector bundles and $K$-equivariant morphisms over $X$, the following statements hold.
\begin{enumerate}
\item  Every short exact sequence splits.
\item  The splitting morphisms of a short exact sequence may be identified with the sections of a vector bundle.
\item  Every vector bundle is a direct summand in a trivial vector bundle.
\item  The Oka-Weil approximation theorem holds for holomorphic sections of a vector bundle.
\end{enumerate}
\end{lemma}

\begin{proof}
(1)  To split a short exact sequence
\[ 0 \to A \to B \to C \to 0 \]
of vector bundles over $X$ in the $K$-category, it suffices that the last term vanish in the exact sequence
\[ \Hom^K(C,B) \to \Hom^K(C,C)\to H^1(X,\mathscr H\!\textit{om}^K(C,A)). \]
The sheaf $r_*^K\mathscr H\!\textit{om}(C,A)$ is coherent on the Stein space $X\git K$ (\cite[p.~27]{HH1999}; see also \cite[Theorem 3.1]{Roberts1986}), where $r$ denotes the quotient map $X\to X\git K$, so
\[ H^1(X,\mathscr H\!\textit{om}^K(C,A))=H^1(X\git K, r_*^K\mathscr H\!\textit{om}(C,A))=0. \]

(2)  Two splitting morphisms $C\to B$ differ by a morphism from $C$ to the kernel $A$ of $B\to C$.  Such morphisms are the sections of a vector bundle.

(3)  We need the bundle to be generated by finitely many $K$-sections.  This follows from the equivariant version of Theorem A, which comes from the equivariant version of Theorem B, which comes from the coherence theorem cited in the proof of (1) and the ordinary Theorem B on the quotient $X\git K$.  To conclude, apply (1). 

(4)  We invoke (3), apply ordinary Oka-Weil, and average over $K$.
\end{proof}

\begin{proof}[Proof of Theorem \ref{t:G-ell-implies-G-Runge}]
Let $Y$ be a $G$-elliptic manifold.  Let $\pi:E\to Y$ be a $G$-vector bundle with a dominating $G$-spray $s:E\to Y$.  Let $X$ be a Stein $G$-manifold and $\Omega$ be a $G$-saturated Runge domain in $X$.  Let $Z=X\times Y$ and let $p:Z\to X$ and $q:Z\to Y$ be the projections.  Maps $X\to Y$ correspond to sections of $p$.  We pull $E$ back to a $G$-vector bundle $q^*\pi:q^*E\to Z$, and obtain a $G$-map $\sigma:q^*E\to Z$, $((x,y),v)\mapsto (x,s(v))$, where $v\in E_y$ ($\sigma$ is called a fibre-dominating $G$-spray; the word \textit{fibre} refers to the fibres of $p$).  The map $\sigma$ takes the fibre of $q^*E$ over $(x,y)\in Z$ nondegenerately into the fibre $p^{-1}(x)$ in $Z$.

The domination property of $s$ means that $Ds:\Ker D\pi\vert Y\to TY$ is an epimorphism of bundles over $Y$, that is, fibrewise surjective.  Here, $\Ker D\pi\vert Y$ is the vertical subbundle of the tangent bundle of $E$, restricted to its zero section, which we identify with $Y$, so $\Ker D\pi\vert Y$ is naturally identified with $E$ itself.  Therefore, $D\sigma:\Ker D q^*\pi\vert Z \to \Ker Dp$ is an epimorphism of bundles over $Z$.

Let $f\in\O^G(S,Y)$, where $S$ is a $G$-saturated Stein domain in $X$, and let $h:x\mapsto (x,f(x))$ be the corresponding section of $p$ over~$S$.  By Siu's Stein neighbourhood theorem, the Stein $G$-submanifold $h(S)$ of $S\times Y$ has a Stein neighbourhood $W'$ in $S\times Y$.  Let $K$ be a maximal compact subgroup of $G$.  Then $W=\bigcap\limits_{k\in K} kW'$ is a $K$-invariant Stein neighbourhood of $h(S)$ in $S\times Y$ (note that the complement of $W$ is closed since $K$ is compact).

By Lemma \ref{l:equivariant-Stein-theory}, over $W$, viewed as a Stein open subset of $Z$, there is a $K$-invariant holomorphic subbundle $F$ of $\Ker D q^*\pi\vert Z$ such that $\Ker D q^*\pi\vert Z=\Ker D\sigma\oplus F$ and $D\sigma:F\to\Ker Dp$ is an isomorphism.  Since $\Ker D q^*\pi\vert Z$ is naturally identified with $q^*E$, we may view $F$ as a subbundle of $q^*E$.  By the inverse function theorem, $\sigma$ maps a neighbourhood of the zero section in $F\vert h(S)$ biholomorphically onto a neighbourhood of $h(S)$ in $Z$.  These neighbourhoods may be taken to be $K$-invariant.

Now let $[0,1] \to\O^G(\Omega,Y)$, $t\mapsto f_t$, be a continuous path such that $f_0\in\overline{\rho_\Omega^X\O^G(X,Y)}$.  Let $h_t:x\mapsto (x,f_t(x))$ be the corresponding sections of $p$.  Let $M\subset\Omega$ be compact.  We need to show that $f_1$ can be uniformly approximated on $M$ by holomorphic $G$-maps $X\to Y$.

A Runge domain in a Stein manifold can be exhausted by relatively compact Runge subdomains.  Here these can be taken to be $K$-invariant, using a suitable $K$-invariant strictly plurisubharmonic exhaustion function.

Choose a $K$-invariant Runge domain $U_0$ with $M\subset U_0\Subset\Omega$.  There is a partition $0=t_0<t_1<\cdots<t_k=1$ of $[0,1]$ such that for each $j=0,\ldots,k-1$ and each $t\in[t_j, t_{j+1}]$, $h_t(U_0)$ lies in a neighbourhood of $h_{t_j}(U_0)$ that is a biholomorphic image by $\sigma$ of a neighbourhood of the zero section in $F_j\vert h_{t_j}(U_0)$.  Here, $F_j$ is the bundle $F$ obtained as above with $S=\Omega$ and $h=h_{t_j}$.  Hence $h_t\vert U_0$ lifts to a uniquely determined holomorphic section $\xi$ of $F_j\vert h_{t_j}(U_0)$ with $\sigma\circ\xi\circ h_{t_j}=h_t$ on $U_0$.  Clearly, $\xi$ is $K$-equivariant.

Choose $K$-invariant Runge domains $U_1,\ldots,U_k$ and $V_0,\ldots,V_{k-1}$ such that
\[ M\subset U_k\Subset V_{k-1}\Subset U_{k-1}\Subset\cdots\Subset U_1\Subset V_0\Subset U_0\Subset \Omega. \]
By assumption, $h_0$ can be uniformly approximated on $U_0$ by a holomorphic $G$-section $g_0$ of $p$ defined on all of $X$.  If the approximation is close enough, then $F_0$ is defined on $g_0(U_0)$ and $h_{t_1}\vert U_0$ lifts to a $K$-section $\xi$ of $F_0\vert g_0(U_0)$ with $\sigma\circ\xi\circ g_0=h_{t_1}$ on $U_0$.  By Oka-Weil approximation theorem for splitting morphisms (Lemma \ref{l:equivariant-Stein-theory}), there is a bundle $F$ obtained as above with $S=X$ and $h=g_0$ that approximates $F_0$ closely enough on $g_0(V_0)$ that $h_{t_1}\vert V_0$ lifts to a section $\xi$ of $F\vert g_0(V_0)$ with $\sigma\circ\xi\circ g_0=h_{t_1}$ on $V_0$.  By Oka-Weil for sections of $F\vert g_0(X)$ (Lemma \ref{l:equivariant-Stein-theory}), $\xi$ can be uniformly approximated on $g_0(U_1)$ by global $G$-sections of $F\vert g_0(X)$.  Thus $h_{t_1}$ can be uniformly approximated on $U_1$ by global $G$-sections of $p$.

Continuing in this way, we see that $h_1$ can be uniformly approximated on $U_k$ by global $G$-sections of $p$.  Hence, $f_1$ can be uniformly approximated on $M$ by holomorphic $G$-maps $X\to Y$.
\end{proof}

We can now conclude that homogeneous spaces are equivariantly Runge.

\begin{corollary}
A complex manifold with a transitive holomorphic action of a reductive complex Lie group $G$ is $G$-Runge.
\end{corollary}

\end{document}